\documentclass[a4paper, 12pt]{article}
\usepackage{amssymb}
\usepackage[english]{babel}
\usepackage{epsfig}
\usepackage{graphicx}
\usepackage{amsmath}
\usepackage{pdfsync}

\large\normalsize \hbadness3000 \vbadness30000
\def\mytitle#1{\setcounter{equation}{0}
\setcounter{footnote}{0}
\begin{center}\Large\textbf{#1}\end{center}
\vspace{0.25cm}}
\def\myname#1{\centerline{{\large #1}}\vspace{-0.13cm}}

\newtheorem{theorem}{Theorem}[section]

\newtheorem{corollary}[theorem]{Corollary}

\newtheorem{definition}[theorem]{Definition}

\newtheorem{lemma}[theorem]{Lemma}

\newenvironment{proof}[1][Proof]{\noindent\textbf{#1: }}{\hspace{\stretch{1}}\rule{0.5em}{0.5em}}

\begin{document}
\mytitle{On Clique Convergences of Graphs }
\myname{S. M. Hegde, V. V. P. R. V. B. Suresh Dara}

\begin{center}
Department of Mathematical and Computational Sciences, \\National Institute of Technology Karnataka, \\Surathkal, Mangalore-575025
\\
{\em E-mail:} smhegde@nitk.ac.in, suresh.dara@gmail.com
\end{center}

\begin{abstract}
Let $G$ be a graph and $\mathcal{K}_G$ be the set of all cliques of $G$, then the clique graph of G denoted by $K(G)$ is the graph with vertex set $\mathcal{K}_G$ and two elements $Q_i,Q_j \in \mathcal{K}_G$ form an edge if and only if $Q_i \cap Q_j \neq \emptyset$. Iterated clique graphs are defined by $K^0(G)=G$, and $K^n(G)=K(K^{n-1}(G))$ for $n>0$. In this paper we determine the number of cliques in $K(G)$ when $G=G_1+G_2$, prove a necessary and sufficient condition for a clique graph $K(G)$ to be complete when $G=G_1+G_2$, give a characterization for clique convergence of the join of graphs and if $G_1$, $G_2$ are Clique-Helly graphs different from $K_1$ and $G=G_1 \Box G_2$, then $K^2(G) = G$.
\end{abstract}



\section{Introduction}

Given a simple graph $G = (V, E)$, not necessarily finite, a clique in $G$ is a maximal complete subgraph in $G$. Let $G$ be a graph and $\mathcal{K}_G$ be the set of all cliques of $G$, then the clique graph operator is denoted by $K$ and the clique graph of $G$ is denoted by $K(G)$. Where $K(G)$ is the graph with vertex set $\mathcal{K}_G$ and two elements $Q_i,Q_j \in \mathcal{K}_G$ form an edge if and only if $Q_i \cap Q_j \neq \emptyset$. Clique graph was introduced by Hamelink in 1968 \cite{hamelink1968partial}.  Iterated clique graphs are defined by $K^0 (G)=G$, and $K^{n} (G)=K(K^{n-1} (G))$ for $n>0$(see \cite{hedetniemi1972line, prisner1995graph, szwarcfiter2003survey}). 

\begin{definition}\label{NCD1}
A graph $G$ is said to be $K$-periodic if there exists a positive integer $n$ such that $G \cong K^n(G)$ and the least such integer is called the $K$-periodicity of $G$, denoted $K$-per $(G)$.
\end{definition}

\begin{definition}\label{NCD2}
A graph $G$ is said to be $K$-Convergent if $\{ K^{n}(G): n \in \mathbb{N}\}$ is finite, otherwise $G$ is $K$-Divergent (see \cite{neumann1978clique}).
\end{definition}

\begin{definition}\label{NCD3}
A graph $H$ is said to be $K$-root of a graph $G$ if $K(H) = G$. 
\end{definition}

If $G$ is a clique graph then one can observe that, the set of all $K$- roots of $G$ is either empty or infinite.

%

\begin{definition}\label{NCD5}\cite{prisner1995graph}
A graph $G$ is a Clique-Helly Graph if the set of cliques has the Helly-Property. That is, for every family of pairwise intersecting cliques of the graph, the total intersection of all these cliques should be non-empty also.
\end{definition}

\begin{definition}\label{NCD6}
Let $G_1=(V_1, E_1)$, $G_2=(V_2, E_2)$ be two graphs. Then their join $G_1 + G_2$ is obtained from the disjoint union by adding all possible edges between vertices of $G_1$ and $G_2$.
\end{definition}

\begin{definition}
The Cartesian product of two graphs $G$ and $H$, denoted $G\Box H$, is a graph with vertex set
$V(G \Box H) = V(G) \times V(H)$,i.e., the set $\{(g,h) | g \in G,h \in H\}$. The edge set of $G\Box H$ consists of all pairs $[(g_1,h_1), (g_2,h_2)]$ of vertices with $[g_1, g_2] \in E(G)$ and $h_1 = h_2$, or $g_1 = g_2$ and $[h_1, h_2] \in E(H)$ (see \cite{imrich2008topics} page no 3).
\end{definition}

In this paper we determine the number of cliques in $K(G)$ when $G=G_1+G_2$, prove a necessary and sufficient condition for a clique graph $K(G)$ to be complete when $G=G_1+G_2$,give a characterization for clique convergence of the join of graphs and if $G_1$, $G_2$ are Clique-Helly graphs different from $K_1$ and $G=G_1 \Box G_2$, then $K^2(G) = G$.

\section{Results}

ne can observe that the clique graph of a complete graph and star graph are always complete. Let $G$ be a graph with $n$ vertices and having a vertex of degree $n-1$, then the clique graph of $G$ is also complete.

\begin{theorem}\label{cgl1}
Let $G_1$, $G_2$ be two graphs and $G=G_1+G_2$, then $X$ is a clique in $G_1$ and $Y$ is clique in $G_2$ if and only if $X+Y$ is a clique in $G_1+G_2$.

\end{theorem}

\begin{proof}
Let  $G=G_1+G_2$ and $X$ be a clique in $G_1$ and $Y$ be a clique in $G_2$. Suppose that $X+Y$ is not a maximal complete subgraph in $G_1+G_2$, then there is a maximal complete subgraph (clique) $Q$ in $G_1+G_2$ such that $X+Y$ is a proper subgraph of $Q$. Since $X+Y$ is a proper subgraph of $Q$, there is a vertex $v$ in $Q$ which is not in $X+Y$ and $v$ is adjacent to every vertex of $X+Y$, then by the definition of $G_1+G_2$, $v$ should be in either $G_1$ or $G_2$. Suppose $v$ is in $G_1$, then the induced subgraph of $V(X) + \{v\}$ is complete in $G_1$, which is a contradiction as $X$ is maximal. Therefore $X+Y$ is the maximal complete subgraph (clique) in $G_1+G_2$.

Conversely, let $Q$ is a clique in $G_1+G_2$. Suppose that $Q \neq X+Y$ where $X$ is a clique in $G_1$ and $Y$ is a clique in $G_2$. If $Q \cap G_1 = \emptyset$, then $Q$ is a subgraph of $G_2$. This implies that $Q$ is a clique in $G_2$ as $Q$ is a clique in $G$. Let $v$ be a vertex of $G_1$. Then by the definition of $G_1+G_2$, one can observe that the induced subgraph of $V(Q) \cup \{v\}$ is complete in $G$, which is a contradiction as $Q$ is a maximal complete subgraph. Therefore $Q \cap G_1 \neq \emptyset$. Similarly we can prove that $Q \cap G_2 \neq \emptyset$. Let $X$ be the induced subgraph of $G$ with vertex set $V(Q) \cap V(G_1) $ and $Y$ be the induced subgraph of $G$ with vertex set $V(Q) \cap V(G_2)$, then $Q=X+Y$. Since $Q$ is a maximal complete subgraph of $G$, $X$ and $Y$ should be maximal complete subgraphs in $G_1$ and $G_2$ respectively. Otherwise, if $X$ is not a maximal complete subgraph in $G_1$ then there is a maximal complete subgraph $X'$ in $G_1$ such that $X$ is subgraph of $X'$, and this implies that $X+Y$ is a subgraph of $X'+Y$ and $X'+Y$ is complete, which is a contradiction. Therefore $X$ and $Y$ are maximal complete subgraphs (cliques) in $G_1$ and $G_2$ respectively.
\end{proof}

\begin{corollary}\label{cgl3}
Let $G_1$, $G_2$ be two graphs and $G=G_1 + G_2$. If $n$, $m$ are the number of cliques in $G_1$, $G_2$ respectively, then $G$ has $nm$ cliques.
\end{corollary}

\begin{proof}
Let $G=G_1+G_2$, $\mathcal{K}_{G_1}=\{X_1, X_2, \dots, X_n\}$ be the set of all cliques of $G_1$ and $\mathcal{K}_{G_2}=\{Y_1, Y_2, \dots, Y_m\}$ be the set of all cliques of $G_2$. Then by Theorem \ref{cgl1} it follows that $\mathcal{K}_G = \{X_i+Y_j: 1\leq i \leq n, 1 \leq j \leq m \}$ is the set of all cliques of $G$. Since $G_1$ has $n$, $G_2$ has $m$ number of cliques, $G_1+G_2$ has $nm$ number of cliques.
\end{proof}

In the following result we give a necessary and sufficient condition for a clique graph $K(G)$ to be complete when $G=G_1+G_2$.

\begin{theorem}\label{cgt2}
Let $G_1$, $G_2$ be two graphs. If $G=G_1+ G_2$, then $K(G)$ is complete if and only if either $K(G_1)$ is complete or $K(G_2)$ is complete.
\end{theorem}

\begin{proof}
Let $G=G_1+ G_2$ and $K(G)$ be complete. Suppose that neither $K(G_1)$ nor $K(G_2)$ is complete, then there exist two cliques $X, X'$ in $G_1$ and two cliques $Y, Y'$ in $G_2$ such that $X \cap X' = \emptyset$ and $Y \cap Y'=\emptyset$. By Theorem \ref{cgl1} it follows that $X+Y, X'+Y'$ are cliques in $G$. Since $X \cap X'$ and $Y \cap Y'$ are empty, it follows that $\{X+Y\} \cap \{X'+Y'\} = \emptyset $, which is a contradiction as $K(G)$ is complete.

Conversely, suppose that $K(G_1)$ is complete and $\mathcal{K}_{G_1}=\{X_1, X_2, \dots, X_n\}$, $\mathcal{K}_{G_2}=\{Y_1, Y_2, \dots, Y_m\}$. By Corollary \ref{cgl3}, it follows that $G$ has exactly $nm$ number of cliques. Let $\mathcal{K}_{G}=\{Q_{ij}:Q_{ij}=X_i+Y_j$ for $ i=1, 2, \dots, n; j=1, 2, \dots, m\}$ be the set of all cliques of $G$. Then $Q$ is the vertex set of $K(G)$. Arranging the elements of $\mathcal{K}_{G}$ in the matrix form $M=[m_{ij}]$ where $m_{ij}=Q_{ij}$, we have

$M=\left(
  \begin{array}{ccccc}
    Q_{11} & Q_{12} & Q_{13} & \dots & Q_{1m} \\
    Q_{21} & Q_{22} & Q_{23} & \dots & Q_{2m} \\
    \vdots & \vdots & \vdots & \ddots & \vdots \\
    Q_{n1} & Q_{n2} & Q_{n3} & \dots & Q_{nm} \\
  \end{array}
\right)$.

Let $Q_{ij}$, $Q_{kl}$ be any two elements in $M$.  Since $Q_{ij}=X_i+Y_j$, $Q_{kl}=X_k+Y_l$, it follows that $X_i$, $X_k$ are cliques in $G_1$. Since $K(G_1)$ is complete, $X_i \cap X_k \neq \emptyset$ and then $Q_{ij} \cap Q_{kl} \neq \emptyset$. Therefore $Q_{ij}$, $Q_{kl}$ are adjacent in $K(G)$. Hence $K(G)$ is complete.
\end{proof}

\begin{lemma}\label{cgl4}
Let $G_1$, $G_2$ be two graphs and $G=G_1+G_2$. If $K(G_1)$, $K(G_2)$ are not complete, then for every clique in $K(G_1)$ there is a clique in $K(G)$.
\end{lemma}

\begin{proof}
Let $G=G_1+G_2$ be a graph such that $K(G_1)$ and $K(G_2)$ are not complete. Let $V(K(G_1))=\{X_i: X_i ${ is a clique in }$ G_1, 1\leq i \leq n \}$ and $V(K(G_2))=\{Y_j: Y_j $t{ is a clique in }$ G_2, 1\leq j \leq m \}$, then by Theorem \ref{cgl1} it follows that  $V(K(G))=\{X_i+Y_j: 1\leq i \leq n, 1\leq j \leq m\}$. Let $Q$ be a clique of size $l$ in $K(G_1)$ and $V(Q)=\{X_{Q_1}, X_{Q_2}, \dots, X_{Q_l}\}$ where $X_{Q_i}$ is a clique in $G_1$ for $1\leq i \leq l$. Let $A_Q = \{X_{Q_i} + Y_j : 1\leq i \leq l, 1 \leq j \leq m\}$. Then clearly $A_Q$ is subset of $V(K(G))$.

Let $X_{Q_1}+Y_1$, $X_{Q_2}+Y_2$ be two elements in $A_Q$. Since $X_{Q_1}, X_{Q_2}$ are the vertices of the clique $Q$ of $K(G_1)$, we have $X_{Q_1} \cap X_{Q_2} \neq \emptyset$. Therefore $\{X_{Q_1}+Y_1\} \cap \{X_{Q_2}+Y_2\} \neq \emptyset$. Hence the intersection of any two elements in $A_Q$ is nonempty. Then, it follows that the elements of $A_Q$ form a complete subgraph in $K(G)$. Suppose that it is not a maximal complete subgraph in $K(G)$. Then there is a vertex, say $X_1+Y_1$ in $K(G)$ which is not in $A_Q$ and $X_1+Y_1$ is adjacent with every vertex of $A_Q$. Since $K(G_2)$ is not complete there exists a vertex say $Y_2$ in $K(G_2)$ such that $Y_2$ is not adjacent to $Y_1$ in $K(G_2)$. Since $Q$ is a clique in $K(G_1)$ and $K(G_1)$ is not complete, there is a vertex say $X_{Q_1}$ in $V(Q)$ which is not adjacent to $X_1$ in $K(G_1)$. By the definition of $A_Q$ one can see that $X_{Q_1} + Y_2$ is an element of $A_Q$. Therefore $\{X_{Q_1} + Y_2\} \cap \{X_1+Y_1\} = \emptyset$, which is a contradiction. Thus $A_Q$ is a maximal complete subgraph in $K(G)$. Hence for every clique in $K(G_1)$ there is a clique in $K(G)$. 
\end{proof}

Similarly for every clique in $K(G_2)$, there is a clique in $K(G)$.

\begin{lemma}\label{cgl5}
Let $G_1$, $G_2$ be two graphs and $G=G_1+G_2$. If $K(G_1)$, $K(G_2)$ are not complete, then for every clique in $K(G)$ there is a clique, either in $K(G_1)$ or in $K(G_2)$ but not in both.
\end{lemma}

\begin{proof}
Let $G=G_1+G_2$ be a graph such that $K(G_1)$ and $K(G_2)$ are not complete. Let $V(K(G_1))=\{X_i: X_i ${ is a clique in }$ G_1, 1\leq i \leq n \}$ and $V(K(G_2))=\{Y_j: Y_j ${ is a clique in }$ G_2, 1\leq j \leq m \}$, then by Theorem \ref{cgl1}, $V(K(G))=\{X_i+Y_j: 1\leq i \leq n, 1\leq j \leq m\}$. One can observe that for every $X_i$ in $V(K(G_1))$, the vertices $X_i+Y_1, X_i+Y_2, \dots, X_i+Y_m$ form a complete graph in $K(G)$, $1 \leq i \leq n$. Similarly, for every $Y_j$ in $V(K(G_2))$, the vertices $X_1+Y_j, X_2+Y_j, \dots, X_n+Y_j$ form a complete graph in $K(G)$, $1 \leq j \leq m$. Therefore every clique in $K(G)$ is of size $ln$ or $lm$.  Let $Q$ be a clique of size $lm$ in $K(G)$ and $V(Q)=\{X_{Q_1}+Y_1, X_{Q_1}+Y_2, \dots, X_{Q_1}+Y_m, X_{Q_2}+Y_1, X_{Q_2}+Y_2, \dots, X_{Q_2}+Y_m, \dots, X_{Q_l}+Y_1, X_{Q_l}+Y_2, \dots X_{Q_l}+Y_m\}$ where $X_{Q_i}$, for $1\leq i \leq l$ is the clique in $G_1$. Define $A_Q = \{X_{Q_i} : 1\leq i \leq l\}$. Clearly $A_Q$ is a subset of $V(K(G_1))$.

Let $X_{Q_1}$ and $X_{Q_2}$ be two elements of $A_Q$. Since $K(G_2)$ is not complete, there exists a vertex, say $Y_2$ in $K(G_2)$ such that $Y_2$ is not adjacent to $Y_1$ in $K(G_2)$, this implies that $Y_1 \cap Y_2 = \emptyset$. Since $X_{Q_1}+Y_1, X_{Q_2}+Y_2$ are the vertices of the clique $Q$ of $G$, $\{X_{Q_1}+Y_1\} \cap \{X_{Q_2}+Y_2\} \neq \emptyset$. Therefore $X_{Q_1} \cap X_{Q_2} \neq \emptyset$. Hence the intersection of any two elements in $A_Q$ is nonempty. It follows that the elements of $A_Q$ form a complete subgraph in $K(G_1)$. Suppose that it is not a maximal complete subgraph in $K(G_1)$, then there is a vertex say $X_1$ in $K(G)$ which is not in $A_Q$ and $X_1$ is adjacent with every vertex in $A_Q$, this implies that $X_1 \cap X_{Q_i} \neq \emptyset$, $1\leq i \leq l$. Since $X_1$ is not in $A_Q$, the vertex $X_1 + Y_1 $ is not in $Q$. Since $X_1$ is adjacent with every element in $A_Q$, $\{X_1+Y_1\} \cap \{X_{Q_i}+Y_j\} \neq \emptyset$ for every $i, j$, $1\leq i \leq l$, $1 \leq j \leq m$. This implies that the vertex $X_1+ Y_1$ is adjacent to every vertex of $Q$ in $K(G)$, which is a contradiction as $Q$ is maximal in $K(G)$. Therefore the elements of $A_Q$ form a maximal complete subgraph (clique) in $K(G_1)$. Hence for every clique of size $lm$ in $K(G)$ there is a clique of size $l$ in $K(G_1)$. Similarly we can prove that if the clique in $K(G)$ is of size $ln$, then there is a clique of size $l$ in $K(G_2)$.
\end{proof}

\begin{lemma}\label{cgl6}
Let $G_1$, $G_2$ be two graphs and $G=G_1+G_2$. If $K(G_1)$, $K(G_2)$ are not complete, then the number of cliques in $K(G)$ is the sum of the number of cliques in $K(G_1)$ and $K(G_2)$.
\end{lemma}

\begin{proof}
Proof of this Lemma follows by the Lemmas \ref{cgl4} and \ref{cgl5}.
\end{proof}

By the definition of a clique graph, cliques of $K(G)$ are the vertices of $K^2(G)$. By Lemmas \ref{cgl4}, \ref{cgl5} and \ref{cgl6} it follows that there is a one to one correspondence between $V(K^2(G))$ and $V(K^2(G_1)) \cup V(K^2(G_2))$ where $G=G_1+G_2$. By the definition of a clique graph, cliques of $G$ are the vertices of $K(G)$. By Corollary \ref{cgl3} it follows that, if $|V(K(G_1))| = n$ and $|V(K(G_2))| = m$ then $|V(K(G))| =nm$. Therefore $K(G) \neq K(G_1) +K(G_2)$.

\begin{theorem}\label{cgt3}
Let $G_1$, $G_2$ be two graphs and $G=G_1+G_2$. If $K(G_1)$, $K(G_2)$ are not complete, then $K^2(G)=K^2(G_1)+K^2(G_2)$.
\end{theorem}

\begin{proof}
Let $G=G_1 + G_2$ be a graph such that $K(G_1)$ and $K(G_2)$ are not complete. Let $X_1, X_2, \dots, X_n$ be the cliques of $K(G_1)$, and $Y_1, Y_2, \dots Y_m$ be the cliques of $K(G_2)$. By Lemma \ref{cgl6}, there are $(n+m)$ cliques in $K(G)$. By Lemma \ref{cgl4} it follows that for every clique $X_i$ of $K(G_1)$ there is a clique $X_i'$ in $K(G)$, $1\leq i \leq n$ and for every clique $Y_j$ of $K(G_2)$ there is a clique $Y_j'$ in $K(G)$, $1 \leq j \leq m$.

Claim 1: $X_i \cap X_j \neq \emptyset$ in $K(G_1)$ if and only if $X_i' \cap X_j' \neq \emptyset$ in $K(G)$ for $i \neq j$.

Let $X_i, X_j$ be two cliques  in $K(G_1)$ and $X_i \cap X_j \neq \emptyset$. Let $v$ be a vertex in $X_i \cap X_j$. By Lemma \ref{cgl4} it follows that if $v$ is a vertex in the clique $X_i$ in $K(G_1)$, then for any vertex $u$ in $K(G_2)$, $v+u$ is a vertex in the clique $X_i'$ in $K(G)$ corresponding to the clique $X_i$ in $K(G_1)$. Therefore $v+u$ is a vertex in $X_i' \cap X_j'$.

Conversely, suppose that $X_i', X_j'$ be two cliques in $K(G)$ and $X_i' \cap X_j' \neq \emptyset$. Let $w$ be a vertex in $X_i' \cap X_j'$. By Lemma \ref{cgl5} it follows that $w = v + u$, where $v$ is a vertex of $K(G_1)$ and $u$ is a vertex of $K(G_2)$. Since $w = v + u$ is a vertex of the clique $X_i'$ in $K(G)$, it follows that $v$ is a vertex of the clique $X_i$ in $K(G_1)$. Similarly $v$ is a vertex of the clique $X_j$ in $K(G_1)$. Therefore $v$ is in $X_i \cap X_j$.

Similarly we can prove that, $Y_i \cap Y_j \neq \emptyset$ in $K(G_2)$ if and only if $Y_i' \cap Y_j' \neq \emptyset$ in $K(G)$ for $i \neq j$.

Claim 2: $X_i' \cap Y_j' \neq \emptyset$ in $K(G)$ for $1\leq i \leq n$, $1 \leq j \leq m$.

Let $X_i', Y_j'$ be two cliques in $K(G)$, $1\leq i \leq n$, $1 \leq j \leq m$ and $X_i, Y_j$ are the cliques in $K(G_1), K(G_2)$ corresponding to the maximal cliques $X_i', Y_j'$ in $K(G)$ respectively. Let $v$ be a vertex in $X_i$ and $u$ be a vertex in $Y_j$, then by Lemma \ref {cgl4} $v+u$ be the vertex in $X_i'$ as well as in $Y_j'$. Therefore $X_i' \cap Y_j' \neq \emptyset$.

Since  cliques of $K(G)$, $K(G_1)$ and $K(G_2)$ are the vertices of $K^2(G)$, $K^2(G_1)$ and $K^2(G_2)$ respectively, by claims 1 and 2 it follows that $K^2(G)$ is the same as $K^2(G_1)+K^2(G_2)$. 
\end{proof}

Let $G_1$, $G_2$ be two graphs, $G=G_1+G_2$ and $K^n(G_1)$, $K^m(G_2)$ are not complete for any $n, m$ in $\mathbb{N}$. Since $K^n(G_1)$, $K^m(G_2)$ are not complete for any $n, m$ in $\mathbb{N}$, $K(G_1)$, $K(G_2)$ are not complete. By Theorem \ref{cgt3}, $K^2(G) = K^2(G_1)+K^2(G_2)$. Since $K^n(G_1)$, $K^m(G_2)$ are not complete for any $n, m$ in $\mathbb{N}$, $K^3(G_1)=K(K^2(G_1))$, $K^3(G_2)=K(K^2(G_2))$ are not complete. Hence by Theorem \ref{cgt3} it follows that $K^2(K^2(G)) = K^2(K^2(G_1))+K^2(K^2(G_2))$. i.e., $K^4(G) = K^4(G_1)+K^4(G_2)$. Proceeding like this we get $K^{2n}(G) = K^{2n}(G_1)+K^{2n}(G_2)$ for any $n$ in $\mathbb{N}$.

\begin{theorem}\label{cgt4}
Let $G_1$, $G_2$ be two graphs and $G=G_1+G_2$. If $K^n(G_1)$, $K^m(G_2)$ are not complete for any $n, m$ in $\mathbb{N}$, then $G$ is $K$-convergent if and only if $G_1, G_2 $ are $K$-convergent.
\end{theorem}

\begin{proof}
Let $G=G_1 + G_2$ be a graph such that $K^n(G_1)$ and $K^m(G_2)$ are not complete for any $n, m$ in $\mathbb{N}$. 

Suppose $G$ is $K$-convergent and $G_1, G_2 $ are not $K$-convergent. By Theorem \ref{cgt3} it follows that $K^{2n}(G) = K^{2n}(G_1)+K^{2n}(G_2)$ for any $n$ in $\mathbb{N}$. Since $G_1, G_2 $ are not $K$-convergent, by definition of convergence, $K^{2n}(G_1)$ and $K^{2n}(G_2)$ are also not $K$-convergent for any $n$ in $\mathbb{N}$. Therefore $K^{2n}(G)$ is not convergent for any $n$ in $\mathbb{N}$ which is a contradiction, as if $G$ is convergent, then $K^n(G)$ is also convergent for any $n$ in $\mathbb{N}$.

Conversely, suppose that $G_1, G_2 $ are $K$-convergent. By Theorem \ref{cgt3} it follows that $K^{2n}(G) = K^{2n}(G_1)+K^{2n}(G_2)$ for any $n$ in $\mathbb{N}$. Since $G_1, G_2 $ are $K$-convergent, by definition of convergence, the sets $\{K^n(G_1) : n \in \mathbb{N} \}$, $\{K^m(G_2) : m \in \mathbb{N} \}$ are finite, which implies that the set $\{K^{2n}(G)=K^{2n}(G_1)+K^{2n}(G_2) : n \in \mathbb{N} \}$ is also finite. i.e., there exists an $n$ in $\mathbb{N}$ such that $K^{2n}(G)=K^{2m}(G)$ for some $m < n$, which implies that the set $\{K^n(G) : n \in \mathbb{N} \}$ is also finite. Therefore $G$ is $K$-convergent.
\end{proof}

\begin{theorem}\label{cgt5}
Let $G_1$ , $G_2$ be two graphs and $G=G_1 + G_2$. If $K^n(G_1)$ is complete for some $n$ in $\mathbb{N}$, then $G$ is $K$-convergent.
\end{theorem}

\begin{proof}
Let $G=G_1 + G_2$ be a graph. Suppose that $K^n(G_1)$ is complete for some $n$ in $\mathbb{N}$. By Theorem \ref{cgt3} it follows that $K^{2n}(G) = K^{2n}(G_1)+K^{2n}(G_2)$ for any $n$ in $\mathbb{N}$. If $n$ is even, it follows that $K^{n}(G) = K^{n}(G_1)+K^{n}(G_2)$. Since $K(K^n(G_1)) = K_1$ is complete, by Theorem \ref{cgt2} it follows that $K(K^{n}(G))$ is also complete. If $n$ is odd, then $n+1$ is even, therefore $K^{n+1}(G) = K^{n+1}(G_1)+K^{n+1}(G_2)$. Since $K^n(G_1)$ is complete, for any $m >n$, $K^{m}(G_1) =K_1$ is complete. By Theorem \ref{cgt2} it follows that $K(K^{n+1}(G))$ is also complete. By the definition of clique convergence it follows that $G$ is $K$-convergent.
\end{proof}

\begin{theorem}\label{cgt10}
Let $G_1$, $G_2$ be $K$-periodic graphs. If $G = G_1 + G_2$, then $G$ is $K$-periodic.
\end{theorem}

\begin{proof}
Let $G = G_1 + G_2$ where $G_1$, $G_2$ are $K$-periodic graphs of periods $n, m$ respectively. Since $G_1, G_2$ are $K$-periodic, neither $K^i(G_1)$ nor $K^j(G_2)$ are complete for any $i, j$. By Theorem \ref{cgt3} it follows that 
\begin{eqnarray*}
K^{2nm}(G) &=& K^{2nm}(G_1) + K^{2nm}(G_2) \\&=& G_1+G_2\\& =& G
\end{eqnarray*}
Therefore $G$ is $K$-periodic.
\end{proof}

\subsection{Observations}

Let $G=G_1+G_2$ be a graph and $\mathcal{K}_{G_1}=\{X_1, X_2, \dots, X_n\}$ be the set of all cliques of $G_1$ and $\mathcal{K}_{G_2}=\{Y_1, Y_2, \dots, Y_m\}$ be the set of all cliques of $G_2$. By Theorem \ref{cgl1}, it follows that $\mathcal{K}_G = \{ Q_{ij} = X_i + Y_j : 1 \leq i \leq n; 1 \leq j \leq m \}$ is the set of all cliques of $G$. Let $v_{ij}$ be the vertex of $K(G)$ corresponding to the clique $Q_{ij}$ of $G$. Arrange the vertices of $K(G)$ as a matrix $M = [m_{ij}]$, where $m_{ij} = v_{ij}$, i.e.,

M=$\left(
  \begin{array}{ccccc}
    v_{11} & v_{12} & v_{13} & \dots & v_{1m} \\
    v_{21} & v_{22} & v_{23} & \dots & v_{2m} \\
    \vdots & \vdots & \vdots & \ddots & \vdots \\
    v_{n1} & v_{n2} & v_{n3} & \dots & v_{nm} \\
  \end{array}
\right)$.

From the above matrix one can observe that the $i^{th}$ row corresponds to the clique $X_i$ of $G_1$ and $j^{th}$ column corresponds to the clique $Y_j$ of $G_2$, $1\leq i \leq n$, $1 \leq j \leq m$. 

Claim 1: Any two elements in the same row or same column in $M$ are adjacent in $K(G)$.

Let $Q_{ij}$, $Q_{ik}$ be any two elements in the $i^{th}$ row. Since $Q_{ij}=X_i+Y_j$, $Q_{ik}=X_i+Y_k$, $Q_{ij} \cap Q_{ik} = X_i \neq \emptyset$. Therefore $Q_{ij}$, $Q_{ik}$ are adjacent in $K(G)$. Similarly any two elements in the same column are adjacent.

Claim 2: If $X_i \cap X_j \neq \emptyset$, then every vertex of $i^{th}$ row is adjacent to every vertex of $j^{th}$ row, $1\leq i \neq j \leq n$.

Let $X_i \cap X_j \neq \emptyset$ and $v_{ik}$, $v_{jl}$ be any two elements of $i^{th}$ and $j^{th}$ rows respectively in $M$. Since $Q_{ik}=X_i+Y_k$, $Q_{jl}=X_j+Y_l$ are the cliques of $G$ corresponding to the vertices $v_{ik}$, $v_{jl}$ of $K(G)$ and $X_i \cap X_j \neq \emptyset$, we have $Q_{ik} \cap Q_{jl} \neq \emptyset$. Therefore $v_{ik}$, $v_{jl}$ are adjacent in $K(G)$.

Similarly if $Y_i \cap Y_j \neq \emptyset$, then every vertex of $i^{th}$ column is adjacent to every vertex of $j^{th}$ column, $1\leq i \neq j \leq m$.

One can see that the following observations will follow from Case 1 and Case 2.

\noindent 1. If $G=G_1 + G_2$, then $K(G)$ is Hamiltonian.

\noindent 2. If $G=G_1 + G_2$, then $K(G)$ is planar if it satisfies one of the following:

i). The number of cliques in $G_1$ and $G_2$ is less than $3$.

ii). If the number of cliques in $G_1$ is $3$, then either $G_2$ is a complete graph or $G_2$ has exactly two cliques and $K(G_1) = \overline{K_3}$, $K(G_2) = \overline{K_2}$.

iii). If the number of cliques in $G_1$ is $4$, then $G_2$ is a complete graph.

\noindent 3. If $G=G_1 + G_2$ and $n, m$ are the number of cliques in $G_1$, $G_2$, then the degree of any vertex in $K(G)$ is either $(n+m-2)+k(n-1)$, or $(n+m-2)+l(m-1)$, $0 \leq k \leq m$ and $0 \leq l \leq n$.

\noindent 4. Let $G_1$, $G_2$ be two graphs and $G= G_1+G_2$,

i) If both $G_1$ and $G_2$ have odd number of cliques, then $K(G)$ is Eulerian. 

ii) If both $G_1$ and $G_2$ have even number of cliques, then $K(G)$ is Eulerian if $K(G_1)$, $K(G_2)$ are Eulerian. 

iii) If $G_1$ has even number of cliques and $G_2$ has odd number of cliques, then $K(G)$ is Eulerian if degree of each vertex in $K(G_1)$ is odd and $K(G_2)$ is totally disconnected.

\section{Cartesian product of graphs}
\begin{theorem}\label{cgt6}
If $G_1$, $G_2$ are Clique-Helly graphs different from $K_1$ and $G=G_1 \Box G_2$, then $K^2(G) = G$.
\end{theorem}

\begin{proof}
Let $G_1$, $G_2$ be Clique-Helly graphs different from $K_1$ and $G=G_1 \Box G_2$.
Let $V(G_1)=\{v_1, v_2, \dots v_{n_1}\}$ and $V(G_2)=\{u_1, u_2, \dots u_{n_2}\}$, then by the definition of $G_1 \Box G_2$, it follows that $V(G) = \{V_{ij} : V_{ij}=(v_i, u_j) ${ where }$ 1\leq i \leq n_1, 1\leq j \leq n_2\}$, $|V(G)|=n_1n_2$. Also, $G$ has $n_2$ copies of $G_1$ (say, $G^1_1, G^2_1, \dots, G^{n_2}_1$) are vertex disjoint induced subgraphs  and $n_1$ copies of $G_2$ (say, $G^1_2, G^2_2, \dots, G^{n_1}_2$) are vertex disjoint induced subgraphs. Clearly one can observe that $V(G^i_2) \cap V(G^j_1) = V_{ij}$, $V_{ij}$ is not in $V(G^n_2)$ and $V(G^m_1)$ for $n\neq i$, $m\neq j$ for all $1\leq i \leq n_1$, $1\leq j \leq n_2$. As $G = G_1 \Box G_2$, we can see that every clique in $G_1$ and $G_2$ are cliques in $G$. Let $\mathcal{K}_{G_1} = \{Q_1, Q_2, \dots, Q_{l_1}\}$ and $\mathcal{K}_{G_2} = \{P_1, P_2, \dots, P_{l_2}\}$, then 

$\mathcal{K}_{G} = \{Q^1_1, Q^1_2, \dots, Q^1_{l_1}, Q^2_1, Q^2_2, \dots, Q^2_{l_1}, \dots Q^{n_2}_1, Q^{n_2}_2, \dots, Q^{n_2}_{l_1},\\ P^1_1, P^1_2, \dots, P^1_{l_2}, P^2_1, P^2_2, \dots, P^2_{l_2}, \dots, P^{n_1}_1, P^{n_1}_2, \dots, P^{n_1}_{l_2}\}$.

Claim 1: For every vertex $V_{ij}$ in $G$ there is a clique in $K(G)$.

Let $V_{ij}$ be a vertex in $G$ for some $i,j$, $1\leq i \leq n_1, 1 \leq j \leq n_2$. Define $A_{ij} = \{Q: V_{ij} \in Q\} \subseteq \mathcal{K}_{G}$. Clearly intersection of any two cliques in $A_{ij}$ is non empty. Therefore the vertices corresponding to these cliques in $K(G)$ form a complete subgraph in $K(G)$. Suppose it is not a maximal complete subgraph in $K(G)$, then there exists a vertex $V$ in $K(G)$ such that $V$ is adjacent to all the vertices of $A_{ij}$. Let $Q_V$ be the clique in $G$ corresponding to the vertex $V$ in $K(G)$. Clearly $V_{ij}$ is not in $Q_V$. Since every clique in $G$ is either a clique in $G_1$ or a clique in $G_2$, assume that $Q_V$ is a clique in $G^j_1$. Let $Q$ be a clique in $G^i_2$ having the vertex $V_{ij}$, then $Q$ is in $A_{ij}$. Since $V(G^i_2) \cap V(G^j_1) = V_{ij}$, $Q$ is a clique in $G^i_2$ and $V_{ij} \in V(Q)$ and $V(Q) \cap V(G^j_1) = V_{ij}$. Which implies that $V(Q) \cap (V(G^j_1) \setminus \{V_{ij}\}) = \emptyset$. Since $V_{ij}$ is not in $Q_V$ and $Q_V$ is a clique in $G^j_1$, $V(Q_V) \subseteq (V(G^j_1)\setminus V_{ij})$. Therefore $V(Q) \cap V(Q_V) = \emptyset$, a contradiction to the fact that $Q_V$ is adjacent to all the vertices of $A_{ij}$ in $K(G)$. Hence the elements of $A_{ij}$ form a clique in $K(G)$.

Claim 2: For any clique $Q$ in $K(G)$, intersection of all the cliques of $G$ corresponding to the vertices of $Q$ is non empty and a singleton.

Let $Q$ be a clique in $K(G)$ and $V(Q) = \{x_1, x_2, \dots x_n\}$. Suppose all $x_k$'s are cliques in $G^j_1$ for some $j$, $1\leq j \leq n_2$, then the intersection of all $x_k$'s is non empty in $G$, where $x_k \in V(Q)$, as $G^j_1$ satisfies clique-helly property. Let $V \in \cap_{x_k \in Q}x_k$, then $V$ is in $G^i_2$ for some $i$, $1\leq i \leq n_1$. Let $P$ be any clique in $G^i_2$  having a vertex $V$, then $P$ intersects with every element of $V(Q)$. Therefore $V(Q)\cup \{P\}$ forms a complete graph in $K(G)$, a contradiction to the assumption that $Q$ is maximal complete subgraph. Thus the elements of $Q$ are the cliques of $G_1$ and cliques of $G_2$. Since $G^j_1$'s are vertex disjoint and $G^i_2$'s are vertex disjoint, any element of $Q$ is either a clique of $G^j_1$ or a clique of $G^i_2$ for some fixed $i, j$, $1\leq i\leq n_1$, $1\leq j \leq n_2$. Let $x_1, x_2, \dots, x_l$ be the cliques of $G^j_1$ and $x_{l+1}, x_{l+2}, \dots, x_{n}$ be the cliques of $G^i_2$. Since $V(G^j_1) \cap V(G^i_2) = V_{ij}$, $x_{l_1} $ is a clique of $ G^j_1$, $x_{l_2}$ is a clique of $G^i_2$ and $V(x_{l_1}) \cap V(x_{l_2}) \neq \emptyset$, $1\leq l_1 \leq l$, $l+1\leq l_2 \leq n$, $V(x_{l_1}) \cap V(x_{l_2}) = V_{ij}$. Which implies that $V_{ij}$ belongs to every $x_k$ in $Q$. Therefore $\cap_{x_k \in Q} x_k =V_{ij}$.

As the cliques of $K(G)$ are the vertices of $K^2(G)$, by Claims 1 and 2 one can see that there is a one to one correspondence between the vertices of $G$ and $K^2(G)$.

Claim 3: Let $U, V$ be any two adjacent vertices in $G$. Then the intersection of the cliques in $K(G)$ corresponding to these vertices is non empty.

Let $U, V$ be any two adjacent vertices in $G$ and $Q_U$, $Q_V$ be the cliques in $K(G)$ corresponding to the vertices $U$, $V$ in $G$ respectively. Since there is an edge between $U$, $V$ in $G$, there exists a clique $Q$ in $G$ such that the vertices $U$, $V$ are in $Q$. By Claims 1 and 2 it follows that the vertices of $Q_U$ in $K(G)$ are the cliques of $G$ having the vertex $U$ in $G$ is in common. Therefore $Q$ is in $V(Q_U)$. Similarly $Q$ is in $V(Q_V)$. Which implies that $Q_U \cap Q_V \neq \emptyset$. Since cliques of $K(G)$ are the vertices of $K^2(G)$, the vertices corresponding to the cliques $Q_U$ and $Q_V$ of $K(G)$ are adjacent in $K^2(G)$.

Claim 4: Let $P$, $Q$ be any two cliques in $K(G)$. If the intersection of $P$ and $Q$ is non empty, then the vertices in $G$ corresponding to these two cliques are adjacent.

Let $P$, $Q$ be any two cliques in $K(G)$, $P \cap Q \neq \emptyset$ and $U$, $V$ be the vertices in $G$ corresponding to the cliques $P$, $Q$ of $K(G)$ respectively. Since $P \cap Q \neq \emptyset$, there exists a vertex $Q_1$ belonging to $V(P) \cap V(Q)$.  By Claims 1 and 2, one can observe that $Q_1$ is a clique in $G$ and $\cap _{P_i\in V(P)}P_i = U$, $\cap _{Q_i\in V(Q)}Q_i = V$. Thus $U$, $V$ belongs to $V(Q_1)$ in $G$. Therefore $U$, $V$ are adjacent in $G$.

 By Claims 3 and 4 it follows that, two vertices are adjacent in $G$ if and only if the corresponding vertices are adjacent $K^2(G)$.

Therefore $K^2(G)$ is the same as $G$, if $G=G_1 \Box G_2$ and $G_1$, $G_2$ are Clique-Helly graphs such that $G_1$, $G_2$ are different from $K_1$.
\end{proof}

\begin{corollary}\label{cgc7}
Let $G_1$, $G_2$ be two graphs and $G=G_1 \Box G_2$. If $G_1$, $G_2$ are Clique-Helly graphs different from $K_1$, then 
\begin{description}
\item[i] $G$ is a Clique-Helly graph.
\item[ii] $G$ is $K$-periodic.
\item[iii] $G$ is $K$-convergent.
\end{description}

\end{corollary}

\bibliographystyle{plain}
\bibliography{mydatabase}

\end{document}